\documentclass{amsart}

\usepackage{amsfonts}
\usepackage{latexsym}
\usepackage{amssymb}
\usepackage{amsmath}
\usepackage{color}
\usepackage{cite}
\usepackage{xcolor}
\usepackage{cleveref}
\usepackage{microtype}
\usepackage{pgfplots}
\usepgfplotslibrary{polar}

\def\relint{\mathop\mathrm{relint}\nolimits}
\def\bd{\mathop\mathrm{bd}\nolimits}
\def\conv{\mathop\mathrm{conv}\nolimits}
\def\aff{\mathop\mathrm{aff}\nolimits}
\def\lin{\mathop\mathrm{lin}\nolimits}

\newcommand{\norm}[1]{ \|#1 \| }
\def\S{\mathbb{S}}
\def\R{\mathbb{R}}
\def\N{\mathbb{N}}
\def\K{\mathcal{K}}

\def\B{\mathbb{B}}
\def\s{\mathbb{S}}
\def\cir{\mathrm{R}}

\def\inr{\mathrm{r}}

\newcommand{\D}{\mathrm{D}}

\def\wid{\mathrm{w}}

\newtheorem{theorem}{Theorem}[section]
\newtheorem{lemma}{Lemma}[section]
\newtheorem{corollary}{Corollary}[section]
\newtheorem{proposition}{Proposition}[section]

\newtheorem{conjecture}{Conjecture}[section]

\numberwithin{equation}{section}

\begin{document}
	
    \title[Additive colourful Carath\'eodory results and radii]{Additive colourful Carath\'eodory type results with an application to radii}

	\author[M. Brugger]{Matthias Brugger}
	\address{Zentrum Mathematik, Technische Universit\"at M\"unchen, Boltzmannstr. 3, 85747 Garching bei M\"unchen, Germany}
	\email{matthias.brugger@tum.de}
	
	\author[M. Fiedler]{Maximilian Fiedler}
	\address{Zentrum Mathematik, Technische Universit\"at M\"unchen, Boltzmannstr. 3, 85747 Garching bei M\"unchen, Germany}
	\email{maximilian.fiedler@tum.de}
	
	\author[B. Gonz\'alez Merino]{Bernardo Gonz\'{a}lez Merino}
	\address{Departamento de Analisis Matemático, Facultad de Matem\'aticas, Universidad de Sevilla, Apdo. 1160, 41080-Sevilla, Spain}
	\email{bgonzalez4@us.es}
	
	\author[A. Kirschbaum]{Anja Kirschbaum}
	\address{Zentrum Mathematik, Technische Universit\"at M\"unchen, Boltzmannstr. 3, 85747 Garching bei M\"unchen, Germany}
	\email{anja.kirschbaum@tum.de}
	
	\subjclass[2010]{Primary 52A20, Secondary 52A40} \keywords{Additive Helly theorem; Additive colourful Carath\'eodory theorem; Convex bodies; Outer and Inner radii; Vector sums; Minkowski addition}
	
	\thanks{The third author is partially supported by
		Fundaci\'{o}n S\'{e}neca, through the Programa de Formaci\'on Postdoctoral de Personal Investigador, project reference 19769/PD/15,
and the Programme in Support of Excellence Groups of the Regi\'{o}n de Murcia, Spain, project reference 19901/GERM/15, and by
MINECO project reference MTM2015-63699-P, Spain.}

	\date{\today}
	\begin{abstract}
		In this paper we study the behavior of the circumradius with respect to the Minkowski addition in generalized Minkowski spaces.
To do so, we solve additive colourful Carath\'eodory type results, under certain equilibria conditions.
		
	\end{abstract}
	\maketitle
	
	\section{Introduction}
	Let us denote by $\K^n$ the set of all $n$-dimensional convex bodies, i.e., convex and compact sets.

	The \emph{circumradius} of a convex body $K$ with respect to a second convex body $C$ is the smallest
	rescalation $\lambda C$ containing a translation of $K$, and is denoted by $\cir(K,C)$. The \emph{inradius}
	of $K$ w.r.t.~$C$ is the largest rescalation $\lambda C$ containing a translation of $K$, and is denoted by $\inr(K,C)$.
	The \emph{diameter} of $K$ w.r.t.~$C$ is the maximum distance between two points of $K$
	measured w.r.t.~$\norm{\cdot}_{(C-C)/2}$, and is denoted by $\D(K,C)$. Finally, the \emph{minimal
    width} of $K$ w.r.t.~$C$ is the smallest distance between two parallel supporting hyperplanes to $K$
	measured w.r.t.~$\norm{\cdot}_{(C-C)/2}$, and is denoted by $\wid(K,C)$. In the Euclidean space $(\R^n,\|\cdot\|_2)$
    with unit ball $\B^n_2:=\{x\in\R^n:\|x\|_2\leq 1\}$, where $\|x\|_2:=\sqrt{x_1^2+\cdots+x_n^2}$, and unit sphere
    $\S^{n-1}=\{x\in\R^n:\|x\|_2=1\}$, we write $\cir(K):=\cir(K,\B^n_2)$, and the same for $\inr(K)$, $\D(K)$, and $\wid(K)$.
	
	The authors of \cite[Theorems 1.1 and 1.2]{GoHC} initiated the study of the behavior of the successive radii (which are generalizations
	of the classical radii above) w.r.t.~the Minkowski (or vectorial) addition in the Euclidean
    space with unit ball $\B^n_2$. In particular, the authors showed that for any two convex bodies $K$ and $L$
	\begin{equation}\label{eq:BasicRelations}
	\begin{array}{cc}
	\cir(K)+\cir(L)\leq \sqrt{2}\cir(K+L), & \inr(K)+\inr(L)\leq\inr(K+L),\\
	\D(K)+\D(L)\leq\sqrt{2}\D(K+L), & \wid(K)+\wid(L)\leq\wid(K+L).
	\end{array}
	\end{equation}
    Other authors have studied the same questions for the different families of successive radii in the Euclidean space \cite{Go14},
    for the mean successive radii \cite{AHS}, for the Firey (or $p$) sum \cite{GoHC14}, or for the Orlicz-Minkowski sum \cite{CYL}.

    In this paper, however, our aim is to focus on the Minkowski addition of convex bodies, and to compute its
    behavior under the circumradius \emph{measured with respect to an arbitrary convex set $C\in\K^n$}. Hence,
    we study inequalities of the form
    \[
    c_1\,\sum_{i=1}^j\cir(K_i,C)\leq\cir(K_1+\cdots+K_j,C)\leq c_2\,\sum_{i=1}^j\cir(K_i,C),
    \]
    where $K_i,C\in\K^n$, $i\in[j]$, $j\in\N$, and $c_1,c_2>0$ are some absolute constants. We can easily deduce
    the upper bound in the inequality above, i.e., for any $K_i,C\in\K^n$, $i\in[j]$, $j\in\N$,
    since $K_i\subset x^i+\cir(K_i,C)C$, for some $x^i\in\R^n$, then
    \[
    K_1+\cdots+K_j\subset (x^1+\cdots+x^j)+(\cir(K_1,C)+\cdots+\cir(K_j,C))C,
    \]
    and thus
    \[
    \cir(K_1+\cdots+K_j,C)\leq\cir(K_i,C)+\cdots+\cir(K_j,C).
    \]
    Therefore, the lower bound is the interesting case, and the main case in our investigations.

    Let $K\in\K^n$. We denote by $\bd(K)$ the \emph{boundary} of $K$. Moreover, letting $p\in\bd(K)$,
    we say that $u\in\R^n\setminus\{0\}$ is an \emph{exterior outer normal} to some $K$ at $p$ if
    $x^Tu\leq  p^Tu$ for every $x\in K$.
    The solution to the first inequality in \eqref{eq:BasicRelations} was achieved by two \emph{main ingredients}. The first of them is the \emph{optimal containment under homothetics condition}
	(cf.~\cite[Thm.~2.3]{BrKo}).
	\begin{proposition}\label{prop:OOC}
		Let $K,C\in\K^n$ be such that $K\subset C$. The following statements are equivalent:
		\begin{enumerate}
			\item $\cir(K,C)=1$.
			\item There exist $p^1,\dots,p^j\in K\cap\bd (C)$, exterior outer normals $u^1,\dots,u^j$ to $C$
			at $p^1,\dots,p^j$, respectively, and scalars $\lambda_1,\dots,\lambda_j>0$, $2\leq j\leq n+1$,
			such that
			\[
			0=\sum_{i=1}^j\lambda_iu^i.
			\]
		\end{enumerate}
        Moreover, if $C=\B^n_2$, the conditions above are also equivalent to
        \begin{enumerate}
        \item[(3)] There exist $p^1,\dots,p^j\in K\cap\S^{n-1}$ and scalars $\lambda_1,\dots,\lambda_j>0$,
        $2\leq j\leq n+1$, such that
			\[
			0=\sum_{i=1}^j\lambda_ip^i.
			\]
        \item[(4)] For every $a\in\R^n\setminus\{0\}$, $K\cap\{x\in\s^{n-1}:x^Ta\geq 0\}\neq\emptyset$.
        \end{enumerate}
	\end{proposition}
	
	The second ingredient is a somewhat hidden but repeatedly used minmax result (see \cite{GoHC}), which can be expressed
	in the following proposition.
	\begin{proposition}\label{prop:Norm2}
		Let $U:=\{u^1,\dots,u^i\}\subset \mathbb S^{n-1}$, $V:=\{v^1,\dots,v^j\}\subset \mathbb S^{n-1}$,
        and $\lambda_1,\dots,\lambda_i,\mu_1,\dots,\mu_j>0$
		be such that
		\[
		0=\sum_{k=1}^i\lambda_ku^k=\sum_{l=1}^j\mu_lv^l.
		\]
		Then there exist $k\in[i]$ and $l\in[j]$ such that
		\[
        \left\|u^k+v^l\right\|_2\geq\sqrt{2}.
        \]
	\end{proposition}
	
    The acquainted reader will quickly realize that Proposition \ref{prop:Norm2} hides an \emph{additive colourful Carath\'eodory (or Helly) type} result. Similar results have been already considered by other authors (cf.~\cite{ABG,BaJC}). Furthermore, the intimate connection between some Helly type results and some radii notions is not that surprising since one can directly link Proposition \ref{prop:OOC} with Helly's theorem via Rubin's lemma,
    and strengthen Proposition \ref{prop:OOC}(2)
    replacing ``$j\leq n+1$" by ``$j\leq h(C)$", where $h(C)$ denotes the well-known \emph{Helly number of $C$} (cf.~\cite{DGK}).

    Our first result generalizes Proposition \ref{prop:Norm2} onto an arbitrary number of a finite amount of subsets.
    Before stating it, we remember that $K,L\subset\R^n$ are \emph{mutually orthogonal} if $x^Ty=0$ for every $x\in K$ and $y\in L$.
\begin{theorem}\label{thm:norm_2}
		Let $U_i:=\{u_1^i,\dots,u_{k_i}^i\}\subset r_i\mathbb S^{n-1}$, $r_i>0$, $i\in[j]$, $j\in\mathbb N$,
        $\lambda^i_1,\dots,\lambda^i_{k_i}>0$, $2\leq k_i\leq n+1$, and $c\in\R^n$, be such that
		\[
		0=\sum_{l=1}^{k_1}\lambda^1_lu^1_l=\cdots=\sum_{l=1}^{k_j}\lambda^j_lu^j_l.
		\]
		Then
		\[
        \max_{l_1,\dots,l_j}\left\|u^1_{l_1}+\cdots+u^j_{l_j}-c\right\|_2\geq\sqrt{r_1^2+\cdots+r_j^2}.
		\]
		Moreover, equality holds if and only if $c=0$, $j\in[n]$, and $U_k$ and $U_l$ are mutually orthogonal,
        for any choice $1\leq k<l\leq j$.
\end{theorem}
    Theorem \ref{thm:norm_2} implies, in particular, the existence of indices $l_i\in[k_i]$, $i\in[j]$, such that
    $\left\|u^1_{l_1}+\cdots+u^j_{l_j}\right\|_2\geq\sqrt{r_1^2+\cdots+r_j^2}$. Let us also note that the
    case $U_i=\{-u^i,u^i\}$ in Theorem \ref{thm:norm_2} follows by a basic averaging argument.

    In this paper we study questions analogous to \eqref{eq:BasicRelations} specially focused on the circumradius functional.
    For any $K,C\in\K^n$, it is known that all radii can be described by means of the outer radius (cf.~\cite{BrKo2}), so in some sense
    we focus in the \emph{first natural step} towards understanding the behavior of the radii functionals with respect
    to the Minkowski addition. Making use of Theorem \ref{thm:norm_2}, we extend \cite[Theorem 1.1, $i=n$]{GoHC}.
    \begin{theorem}\label{thm:radii_addition_2}
    Let $K_i\in\K^n$, $i \in [j]$, $j\in\mathbb N$. Then
		\[
		\cir(K_1)+\cdots+\cir(K_j)\leq\sqrt{j}\cir(K_1+\cdots+K_j).
		\]
    Moreover, equality holds if $K_k$ and $K_l$ are contained in orthogonal linear subspaces,
    for any choice $1\leq k<l\leq j$, and $\cir(K_1)=\cdots=\cir(K_j)$.
    \end{theorem}

        Let us observe that the \emph{if} in the equality case above is not \emph{only if}. For instance,
        the planar sets
        \[
        K:=\conv(\{(\pm 1,0)^T,(0,\pm(\sqrt{2}-1))^T\}),\quad L:=\conv(\{(0,\pm 1)^T,(\pm(\sqrt{2}-1),0)^T\}),
        \]
        fulfill $\cir(K)=\cir(L)=1$ and $\cir(K+L)=\sqrt{2}$,  but $K$ and $L$ are not mutually orthogonal.

        Theorem \ref{thm:norm_2} presents an optimal estimate only if $U_1,\dots,U_j$ are \emph{no more} than $j\leq n$ sets.
        If we are given more than $n$ of those sets, then the optimal analogous result to Theorem \ref{thm:norm_2} turns out
        to be more involved. In this regard, we have been able to show the following additive Helly type result, which
        improves Theorem \ref{thm:norm_2} in the case of $n=2$, $j=3$, and $r_i=1$, $i\in[3]$,
        from $\sqrt{3}$ to $2$, which is the optimal value in this case.

     \begin{theorem}\label{thm:norm_2_3sets}
     Let $U_i:=\{u_1^i,\dots,u_{k_i}^i\}\subset \S^1$, $i \in [3]$,
     $\lambda^i_1,\dots,\lambda^i_{k_i}>0$, $2\leq k_i\leq 3$, and $c\in\R^2$, be such that
		\[
		0=\sum_{l=1}^{k_1}\lambda^1_lu^1_l=\sum_{l=1}^{k_2}\lambda^2_lu^2_l=\sum_{l=1}^{k_3}\lambda^3_lu^3_l.
		\]
		Then
		\[
    \max_{l_1,l_2,l_3}\left\|u^1_{l_1}+u^2_{l_2}+u^3_{l_3}-c\right\|_2 \geq 2.
		\]
		Moreover, equality holds if and only if $c=0$ and, after a suitable common rotation, we have that
        \[
        U_i=\{\pm(\cos(i\pi/3),\sin(i\pi/3))\}, \quad i \in [3].
        \]
     \end{theorem}
     We would like to point out a fundamental difference between the equality cases of
     Theorems \ref{thm:norm_2} and \ref{thm:norm_2_3sets}, which reflects a reason why the latter is more
     complicated than the former. On the one hand, if $U_1,\dots,U_j$ achieve equality in Theorem \ref{thm:norm_2},
     then \emph{any} choice of vectors $u^i_{l_i}\in U_i$, $i \in [j]$, $j\in [n]$, provides the optimal
     bound $\left\|u^1_{l_1}+\cdots+u^j_{l_j}\right\|_2=\sqrt{r_1^2+\cdots+r_j^2}$. On the other hand,
     \emph{some} choices in the sets of pairs of vectors $\{\pm(\cos(i\pi/3),\sin(i\pi/3))\}$, $i \in [3]$
     (cf.~Theorem \ref{thm:norm_2_3sets}), may lead to \emph{worse values} than the optimal one, for instance,
     \[
     \begin{split}
     (\cos(\pi/3),\sin(\pi/3))^T-(\cos(2\pi/3),\sin(2\pi/3))^T+(\cos(\pi),\sin(\pi))^T& \\
     =(1/2,\sqrt{3}/2)^T-(-1/2,\sqrt{3}/2)^T+(-1,0)^T=(0,0)^T.&
     \end{split}
     \]

     Again, a consequence of Theorem \ref{thm:norm_2_3sets} is the following result.
     \begin{corollary}\label{cor:radii_2_3sets}
     Let $K_i\in\K^2$, $i \in [3]$, be such that $K_i\subset\B^2_2$ with
     $\cir(K_i)=1$, for $i \in [3]$. Then
     \[
     \cir(K_1+K_2+K_3)\geq 2.
     \]
     Moreover, equality holds only if, after a suitable common rotation, $K_i$ contains a segment $[(\cos(i\pi/3),\sin(i\pi/3)),-(\cos(i\pi/3),\sin(i\pi/3))]$, $i \in [3]$.
     \end{corollary}

     Very recently, special attention has been paid to the radii functionals measured with respect to an arbitrary
     $C\in\K^n$ (cf.~\cite{BrGo17,BrKo}). This motivated us to study the behavior of the
     outer radius $\cir(\cdot,C)$ of the sum of a finite amount of convex bodies. For $K_1,\dots,K_m\subset\R^n$,
     let $K_1+\cdots+\widehat{K_i}+\cdots+K_m$ be the sum of all sets $K_j$ such that $j\in[m]\setminus\{i\}$.

     \begin{theorem}\label{thm:CircAdditionC}
     Let $K_i,C\in\K^n$, $i \in [j]$, $j\in\N$. Then
		\begin{equation} \label{eq:equality}
			\cir(K_1,C)+\cdots+\cir(K_j,C)\leq j\,\cir(K_1+\cdots+K_j,C).
		\end{equation}

        Equality holds if and only if there exist polyhedral cylinders
        $C_i$ with facets parallel to
        $\aff(K_1+\cdots+\widehat{K_i}+\cdots+K_j)$ and some $\lambda>0$ and $z\in\R^n$ such that
        \[
        K_1+\cdots+K_j\subset z+\lambda C\subset C_1\cap\cdots\cap C_j
        \]
        with $\cir(K_1,C_1)=\cdots=\cir(K_j,C_j)=1$.
     \end{theorem}

     Before going on, we will now explain the notation used in the remaining sections of the paper.
     For any $A,B\subset\R^n$, we write $A\bot B$ if $A$ and $B$ are mutually orthogonal.
     For any set $A\subset\R^n$, we denote by $\conv(A)$, $\lin(A)$, and $\aff(K)$, the
     \emph{convex}, \emph{linear}, and \emph{affine hull} of $A$, respectively. We denote
     by $\dim(A)$ the \emph{dimension} of $A$, and it is defined as $\dim(A):=\dim(\aff(A))$.
     For any $A,B\subset\R^n$, we say that $A\subset_tB$ (resp.~$A\not\subset_tB$) if there exists
     $x\in\R^n$ such that $A\subset x+B$ (resp.~$A\not\subset x+B$ for every $x \in \R^n$).

	\section{Euclidean case}
	
    We start this section proving Theorem \ref{thm:norm_2}.

	\begin{proof}[Proof of Theorem \ref{thm:norm_2}]
		Let $\lambda^i_1,\dots,\lambda^i_{k_i}>0$ be such that
		$0=\sum_{l=1}^{k_i}\lambda^i_lu_l^i$ for all $i \in [j]$. Considering the scalar product
		\[
		\left(\sum_{l=1}^{k_1}\lambda^1_lu_l^1\right)^Tc=\sum_{l=1}^{k_1}\lambda^1_l(u^1_l)^Tc=0,
		\]
		then there exists $m_1\in[k_1]$ such that $(u^1_{m_1})^Tc\leq 0$.
		Secondly, and again due to
		\[
		\left(\sum_{l=1}^{k_2}\lambda^2_lu_l^2\right)^T(u^1_{m_1}-c)=\sum_{l=1}^{k_2}\lambda^2_l(u^2_l)^T(u^1_{m_1}-c)=0
		\]
		there exists $m_2\in[k_2]$ such that $(u^2_{m_2})^T(u^1_{m_1}-c)\geq 0$.
		In general and because of the same reason, for every $t=3,\dots,j$, we choose $m_t\in[k_t]$ such that
		\[
		(u^t_{m_t})^T(u^1_{m_1}+\cdots+u^{t-1}_{m_{t-1}}-c)\geq 0.
		\]
		Hence, for every $t=2,\dots,j$, we have that
		\[
		\begin{split}
		\norm{u^1_{m_1} & +\cdots+u^t_{m_t}-c}_2^2 \\
& =\norm{u^1_{m_1}+\cdots+u^{t-1}_{m_{t-1}}-c}_2^2+2(u^1_{m_1}+\cdots+u^{t-1}_{m_{t-1}}-c)^T(u^t_{m_t})+\norm{u^t_{m_t}}_2^2 \\
		& \geq \norm{u^1_{m_1}+\cdots+u^{t-1}_{m_{t-1}}-c}_2^2+\norm{u^t_{m_t}}_2^2,
		\end{split}
		\]
		and thus
		\[
		\begin{split}
		\norm{u^1_{m_1}+\cdots+u^j_{m_j}-c}_2^2& \geq\norm{c}_2^2+\norm{u^1_{m_1}}_2^2+\cdots+\norm{u^j_{m_j}}_2^2\\
		& \geq\norm{u^1_{m_1}}_2^2+\cdots+\norm{u^j_{m_j}}_2^2.
		\end{split}
		\]
		Therefore
		\[
		\begin{split}
		\max_{l_1,\dots,l_j}\norm{u^1_{l_1}+\cdots+u^j_{l_j}-c}_2^2 & \geq \norm{u^1_{m_1}+\cdots+u^j_{m_j}-c}_2^2\\
		& \geq \norm{u^1_{m_1}}_2^2+\cdots+\norm{u^j_{m_j}}_2^2 \\
		& = r_1^2+\cdots+r_j^2.
		\end{split}
		\]
		
		We now show the equality case. Since we must have equality in all inequalities above,
		we start observing that $\norm{c}_2=0$ implies $c=0$. Moreover, we also have that
		\[
		(u^t_{l_t})^T(u^1_{l_1}+\cdots+u^{t-1}_{l_{t-1}})=0
		\]
		for every $t=2,\dots,j$, and every $l_t \in [k_t]$. This gives,
		for $t=2$, that $(u^2_{l_2})^Tu^1_{l_1}=0$ for every $l_1,l_2$, and thus
		$U_1\bot U_2$. For $t=3$, we get $(u^3_{l_3})^T(u^1_{l_1}+u^{2}_{l_2})=0$ for every
		$l_1,l_2,l_3$. Clearly
		\[
		\mathrm{lin}(\{u^1_{l_1}+u^2_{l_2}:1\leq l_i\leq k_i\})=\mathrm{lin}(U_1\cup U_2) = \mathrm{lin}(U_1) +  \mathrm{lin}(U_2)
		\]
		and thus $U_3\bot (U_1 + U_2)$. Analogously, we obtain that for every $t=2,\dots,j$
		then $U_t\bot(U_1+\cdots + U_{t-1})$, i.e., $U_i\bot U_l$ for every $1\leq i<l\leq j$.
		Finally, we also observe that $U_i\bot U_l$ for every $1\leq i<l\leq j$
		implies that
		\[
		\mathrm{dim}(U_1)+\cdots+\mathrm{dim}(U_j)=\mathrm{dim}(U_1+\cdots+U_j)\leq n
		\]
		which, since $\mathrm{dim}(U_i)\geq 1$ for $i\in [j]$, hence $j\leq n$, concludes the equality case.
	\end{proof}
	
    Making use of Theorem \ref{thm:norm_2}, we are able to show the following result, which is a nonlinear
    sharp bound explaining the behavior of the Euclidean circumradius with respect to the Minkowski addition of convex sets.

	\begin{theorem}\label{thrm:radii_2}
		Let $K_i\in\K^n$, $i\in[j]$, $j\in\mathbb N$. Then
		\[
		\cir(K_1)^2+\cdots+\cir(K_j)^2\leq\cir(K_1+\cdots+K_j)^2.
		\]
        Equality holds if $K_k$ and $K_l$ are mutually orthogonal,
        for any choice $1\leq k<l\leq j$.		
	\end{theorem}
	
	\begin{proof}
		Let $c_i\in K_i$ be such that $K_i\subset c_i+\cir(K_i)\B^n_2$, $i\in[j]$, and define
		$K_i':=K_i-c_i$. Letting $r_i:=\cir(K_i)$ for $i\in[j]$, by Proposition \ref{prop:OOC}, there exist
		\[
		u_1^i,\dots,u_{k_i}^i\in r_i\mathbb S^{n-1}\cap K_i',
		\]
		such that $0\in\conv(\{u_1^i,\dots,u_{k_i}^i\})$. Let us denote by $U_i:=\{u_1^i,\dots,u_{k_i}^i\}$
		for every $i\in[j]$. By Theorem \ref{thm:norm_2} we then have that
		\[
		\begin{split}
		\cir(K_1+\cdots+K_j) & \geq \cir(\conv(U_1)+\cdots+\conv(U_j))\\
        & \geq \cir(U_1+\cdots+U_j)\\
		& = \min_{c\in\R^n}\max_{l_1,\dots,l_j}\left\|u^1_{l_1}+\cdots+u^j_{l_j}-c\right\|_2 \\
		& \geq\sqrt{r_1^2+\cdots+r_j^2}.
		\end{split}
		\]
        For the equality case, let us observe that if $K_k$ and $K_l$ are mutually orthogonal, for any choice
        $1\leq k<l\leq j$, then $U_k$ and $U_l$ are mutually orthogonal as well, for any $1\leq k<l\leq j$,
        and thus the equality cases of Theorem \ref{thm:norm_2} implies the result.

	\end{proof}
	
	\begin{proof}[Proof of Theorem \ref{thm:radii_addition_2}]
		Applying H\"older's inequality to $\left(\cir(K_1),\dots,\cir(K_j)\right)^T$
		and $(1,\dots,1)^T \in \mathbb{R}^j$ yields
		\begin{align*}
		\sum_{i=1}^j\cir(K_i)
		&= \norm{(\cir(K_1),\dots,\cir(K_j))^T}_1 \\
		&\le \norm{(1,\dots,1)^T}_{2}  \; \norm{(\cir(K_1),\dots,\cir(K_j))^T}_2 \\
		&= j^{\frac{1}{2}}  \norm{(\cir(K_1),\dots,\cir(K_j))^T}_2 \\
		&\le j^{\frac{1}{2}}  \cir(K_1+\cdots+K_j)
		\end{align*}
		where the last inequality follows from Theorem \ref{thrm:radii_2}.
		
        Let us assume that we have equality. Then we have equality in Theorem \ref{thrm:radii_2}. Moreover,
        we have equality in H\"older's inequality applied to $\left(\cir(K_1),\dots,\cir(K_j)\right)^T$
		and $(1,\dots,1)^T$, hence $\cir(K_1)=\cdots=\cir(K_j)$, concluding the proof.

	\end{proof}

\begin{figure}
	\centering
	\begin{tikzpicture}[>=latex]
	
	\def\radius{3}
	
	\foreach \ang/\lab/\dir in {
		0/0/right,
		1/{\pi/6}/{right},
		2/{2\pi/6}/{above},
		3/{3\pi/6}/{above},
		4/{4\pi/6}/{above},
		5/{5\pi/6}/{left},
		6/{6\pi/6}/{left},
		7/{7\pi/6}/{left},
		8/{8\pi/6}/{below},
		9/{9\pi/6}/{below},
		10/{10\pi/6}/{below},
		11/{11\pi/6}/{right}}
	{
		\draw[lightgray] (0,0) -- (\ang * 180 / 6:\radius + 0.1);
		\node [fill=white] at (\ang * 180 / 6:\radius + 0.2) [\dir] {$\lab$};
	}
	
	\draw [] (0,0) circle (\radius);
	
	\draw[->, line width = 1] (0,0) -- node[fill=white] {$e_2$} (90:\radius) ;

	\draw [red,thick,domain=60:120] plot ({\radius*cos(\x)}, {\radius*sin(\x)});
	\node [red] at (90:\radius+1) {$B(e_2,\pi/6)$};
	
	\draw [black!50!green,thick,domain=0:60] plot ({\radius*cos(\x)}, {\radius*sin(\x)});
	\node [black!50!green] at (30:\radius+2) {$B((\sqrt{3}/2,1/2),\pi/6)$};
	
	\draw [blue,thick,domain=120:180] plot ({\radius*cos(\x)}, {\radius*sin(\x)});
	\node [blue] at (150:\radius+2) {$B((-\sqrt{3}/2,1/2),\pi/6)$};

	\end{tikzpicture}
	\caption{Arcs  $B(e_2,\pi/6)$,
		$B((-\sqrt{3}/2,1/2),\pi/6)$ and $B((\sqrt{3}/2,1/2),\pi/6)$ in the proof of Theorem \ref{thm:norm_2_3sets}.}
	\label{fig:angleregions}
\end{figure}
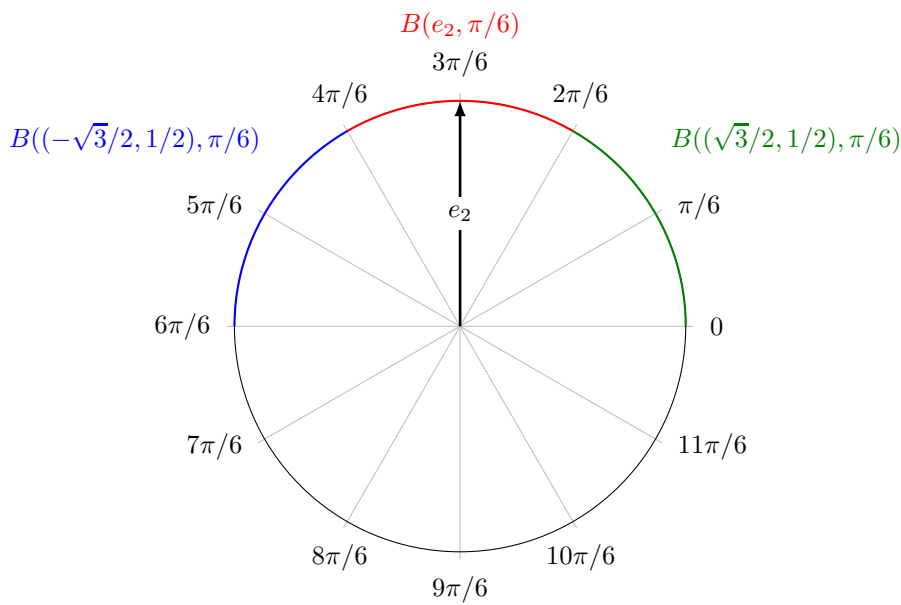

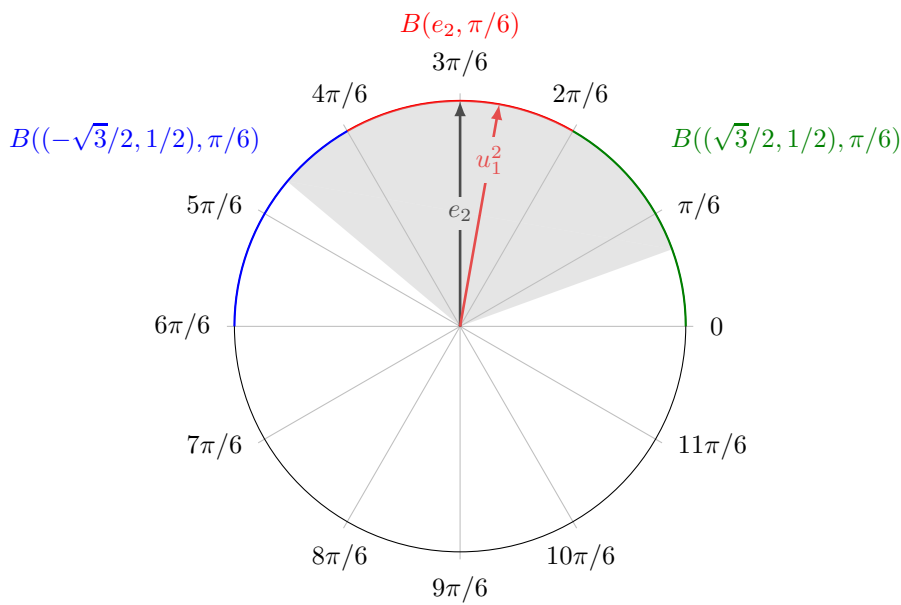
\begin{figure}
	\centering
	\begin{tikzpicture}[>=latex]
	
	\def\radius{3}
	
	\foreach \ang/\lab/\dir in {
		0/0/right,
		1/{\pi/6}/{right},
		2/{2\pi/6}/{above},
		3/{3\pi/6}/{above},
		4/{4\pi/6}/{above},
		5/{5\pi/6}/{left},
		6/{6\pi/6}/{left},
		7/{7\pi/6}/{left},
		8/{8\pi/6}/{below},
		9/{9\pi/6}/{below},
		10/{10\pi/6}/{below},
		11/{11\pi/6}/{right}}
	{
		\draw[lightgray] (0,0) -- (\ang * 180 / 6:\radius + 0.1);
		\node [fill=white] at (\ang * 180 / 6:\radius + 0.2) [\dir] {$\lab$};
	}
	
	\draw [] (0,0) circle (\radius);
	
	\draw[->, line width = 1] (0,0) -- node[fill=white] {$e_2$} (90:\radius) ;

	\draw [red,thick,domain=60:120] plot ({\radius*cos(\x)}, {\radius*sin(\x)});
	\node [red] at (90:\radius+1) {$B(e_2,\pi/6)$};
	\draw[->, red, line width = 1] (0,0) -- node[fill=white, near end] {$u_1^2$} (80:\radius) ;
	\fill [lightgray,thick,domain=20:140, opacity = 0.4] plot ({\radius*cos(\x)}, {\radius*sin(\x)});
	\fill [lightgray,thick,opacity = 0.4] (0,0) -- (20:\radius) -- (140:\radius) -- (0,0);
	
	\draw [black!50!green,thick,domain=0:60] plot ({\radius*cos(\x)}, {\radius*sin(\x)});
	\node [black!50!green] at (30:\radius+2) {$B((\sqrt{3}/2,1/2),\pi/6)$};
	
	\draw [blue,thick,domain=120:180] plot ({\radius*cos(\x)}, {\radius*sin(\x)});
	\node [blue] at (150:\radius+2) {$B((-\sqrt{3}/2,1/2),\pi/6)$};

	\end{tikzpicture}
	\caption{All vectors in the gray region, together with $u_1^2$, fulfill the requirements of (I).}
	\label{fig:angleregions2}
\end{figure}

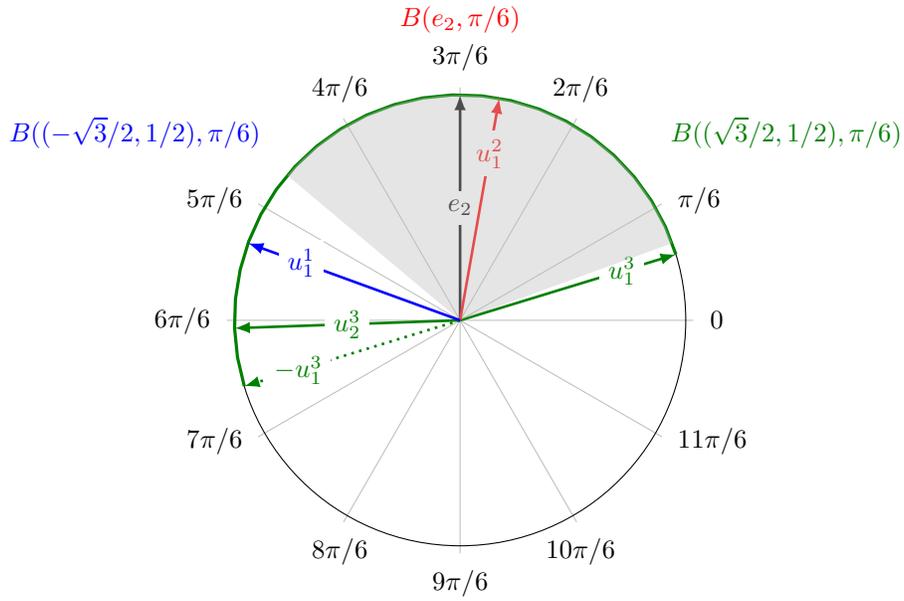
\begin{figure}
	\centering
	\begin{tikzpicture}[>=latex]
	
	\def\radius{3}
	
	\foreach \ang/\lab/\dir in {
		0/0/right,
		1/{\pi/6}/{right},
		2/{2\pi/6}/{above},
		3/{3\pi/6}/{above},
		4/{4\pi/6}/{above},
		5/{5\pi/6}/{left},
		6/{6\pi/6}/{left},
		7/{7\pi/6}/{left},
		8/{8\pi/6}/{below},
		9/{9\pi/6}/{below},
		10/{10\pi/6}/{below},
		11/{11\pi/6}/{right}}
	{
		\draw[lightgray] (0,0) -- (\ang * 180 / 6:\radius + 0.1);
		\node [fill=white] at (\ang * 180 / 6:\radius + 0.2) [\dir] {$\lab$};
	}
	
	\draw [] (0,0) circle (\radius);
	
	\draw[->, line width = 1] (0,0) -- node[fill=white] {$e_2$} (90:\radius) ;
	
	\draw[->, black!50!green, line width = 1] (0,0) -- node[fill=white, near end] {$u_1^3$} (17:\radius);
	\draw[->, dotted, black!50!green, line width = 1] (0,0) -- node[fill=white, near end] {$-u_1^3$} (197:\radius);
	\draw[->, black!50!green, line width = 1] (0,0) -- node[fill=white] {$u_2^3$} (182:\radius);
	\node [black!50!green] at (30:\radius+2) {$B((\sqrt{3}/2,1/2),\pi/6)$};
	\draw [black!50!green,very thick,domain=17:197] plot ({\radius*cos(\x)}, {\radius*sin(\x)});
	
	\node [red] at (90:\radius+1) {$B(e_2,\pi/6)$};
	\draw[->, red, line width = 1] (0,0) -- node[fill=white, near end] {$u_1^2$} (80:\radius) ;
	\fill [lightgray,thick,domain=20:140, opacity = 0.4] plot ({\radius*cos(\x)}, {\radius*sin(\x)});
	\fill [lightgray,thick,opacity = 0.4] (0,0) -- (20:\radius) -- (140:\radius) -- (0,0);

	\draw[->, blue, line width = 1] (0,0) -- node[fill=white, near end] {$u_1^1$} (160:\radius) ;
	\node [blue] at (150:\radius+2) {$B((-\sqrt{3}/2,1/2),\pi/6)$};

	\end{tikzpicture}
	\caption{Since $0$ is contained in the convex hull of $U^3$ there must be a point $u_2^3$ on the green arc.}
	\label{fig:angleregions3}
\end{figure}

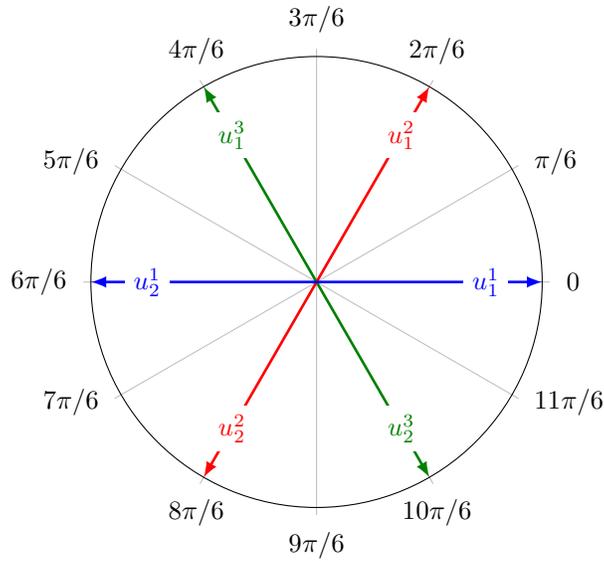
\begin{figure}
	\centering
	\begin{tikzpicture}[>=latex]
	
	\def\radius{3}
	
	\foreach \ang/\lab/\dir in {
		0/0/right,
		1/{\pi/6}/{right},
		2/{2\pi/6}/{above},
		3/{3\pi/6}/{above},
		4/{4\pi/6}/{above},
		5/{5\pi/6}/{left},
		6/{6\pi/6}/{left},
		7/{7\pi/6}/{left},
		8/{8\pi/6}/{below},
		9/{9\pi/6}/{below},
		10/{10\pi/6}/{below},
		11/{11\pi/6}/{right}}
	{
		\draw[lightgray] (0,0) -- (\ang * 180 / 6:\radius + 0.1);
		\node [fill=white] at (\ang * 180 / 6:\radius + 0.2) [\dir] {$\lab$};
	}
	
	\draw [] (0,0) circle (\radius);

	\draw[->, black!50!green, line width = 1] (0,0) -- node[fill=white, near end] {$u_1^3$} (120:\radius);
	\draw[->, black!50!green, line width = 1] (0,0) -- node[fill=white, near end] {$u_2^3$} (300:\radius);

	\draw[->, red, line width = 1] (0,0) -- node[fill=white, near end] {$u_1^2$} (60:\radius) ;
	\draw[->, red, line width = 1] (0,0) -- node[fill=white, near end] {$u_2^2$} (240:\radius) ;

	\draw[->, blue, line width = 1] (0,0) -- node[fill=white, near end] {$u_1^1$} (0:\radius) ;
	\draw[->, blue, line width = 1] (0,0) -- node[fill=white, near end] {$u_2^1$} (180:\radius) ;

	\end{tikzpicture}
	\caption{Equality case in the proof of Theorem \ref{thm:norm_2_3sets} up to a rotation.}
	\label{fig:angleregions4}
\end{figure}

\begin{proof}[Proof of Theorem \ref{thm:norm_2_3sets}]
  We start by showing the following claim.\\
  \emph{Claim.} There exist two points, each from a different group
	of points $U_i$, $i \in [3]$, say w.l.o.g.~$u^1_1$ and $u^2_1$,
	such that
	\begin{equation*}
	(u^1_1+u^2_1)^T(-c)\geq 0\quad\text{and}\quad (u^1_1)^Tu^2_1\geq\frac12,
	\end{equation*}
  cf.~Figure \ref{fig:angleregions3}. Geometrically speaking, the claim states that $(u^1_1+u^2_1)$ has a positive component in the direction of $-c$ and that the angle between $u^1_1$ and $u^2_1$ is no greater than $\frac{\pi}{3}$.

  Before proving claim, let us observe that if the claim holds true,
  since $0=\sum_{l=1}^{k_3}\lambda^3_lu^3_l$,
	then $0=\sum_{l=1}^{k_3}\lambda^3_l(u^1_1+u^2_1-c)^Tu^3_l$. From this we deduce that
	there exists an index $l$, say $l=1$,
	such that $(u^1_1+u^2_1-c)^Tu^3_1\geq 0$. Hence
	\begin{equation}
	\begin{alignedat}{3}
	& &&\|u^1_1+u^2_1+u^3_1-c\|_2^2 \\
	&= &&\|u^1_1\|_2^2 + \|u^2_1\|_2^2  + 2(u^1_1)^Tu^2_1 + \|c\|_2^2 + 2(u^1_1+u^2_1)^T(-c) \\
	&& + &\|u^3_1\|_2^2 +  2(u^1_1+u^2_1-c)^Tu^3_1\\
	& \geq && 1  +  1 + 2\frac12  +  0  +  0  +  1  +  0=4,
	\end{alignedat}
    \label{eq:ineq_case}
	\end{equation}
	concluding the proof of the inequality.
	Hence, we just need to show the claim.

  \begin{proof}[Proof of Claim]
  After a suitable rotation, we can assume
	that $-c=\lambda e_2$, for some $\lambda\geq0$. For the sake of clearness,
	we use the notation $B(u,\alpha):=\{x\in\mathbb S^1:\arccos(u^Tx)\leq\alpha\}$, where $u\in\mathbb S^1$ and $\alpha\in[0,\pi]$. %
  \footnote{Observe that for $u \in \mathbb S^1$, $B(u,\alpha)$ is an arc of $\mathbb S^1$ and contains all vectors that span an angle of no more than $\alpha$ with $u$.}
	We first notice that if there exist two points from different sets of vectors $U_1$, $U_2$, $U_3$,
	namely $u^1_1$ and $u^2_1$, such that $u^1_1,u^2_1\in B(e_2,\pi/6)$ or
	$u^1_1,u^2_1\in B((-\sqrt{3}/2,1/2),\pi/6)$ or $u^1_1,u^2_1\in B((\sqrt{3}/2,1/2),\pi/6)$ then we
	have that
	\[
	(u^1_1+u^2_1)^Te_2=(u^1_1)^Te_2+(u^2_1)^Te_2\geq 0+0=0\quad\text{and}\quad \arccos((u^1_1)^Tu^2_1)\leq\frac\pi3,
	\]
	i.e., $(u^1_1)^Tu^2_1\geq \frac12$, as desired (see Figure \ref{fig:angleregions} for the corresponding arcs).
	
  Let us suppose this were not the case. Since $B(e_2,\pi/6)$
	and $B((\pm\sqrt{3}/2,1/2),\pi/6)$ cover $C:=\{x\in\s^1:x_2\geq 0\}$,
	and $U_1,U_2$, and $U_3$ each intersect with
	$C$ (cf.~Proposition \ref{prop:OOC}(4)), we can suppose w.l.o.g.~that
	$U_1\cap C\subset B((-\sqrt{3}/2,1/2),\pi/6)$,
	$U_2\cap C\subset B((0,1),\pi/6)$, and
	$U_3\cap C\subset B((\sqrt{3}/2,1/2),\pi/6)$.
	Let us also suppose that $u^1_1\in B((-\sqrt{3}/2,1/2),\pi/6)$,
	$u^2_1\in B((e_2,\pi/6)$, and $u^3_1\in B((\sqrt{3}/2,1/2),\pi/6)$.
	Moreover, we can also suppose that $u^1_1$ (resp.~$u^3_1$)
	is the point from $U_1$ (resp.~$U_3$) contained in $B((-\sqrt{3}/2,1/2),\pi/6)$
	(resp.~$B((\sqrt{3}/2,1/2),\pi/6)$) with biggest second coordinate.
	Furthermore, we can suppose that $u^1_1,u^3_1\notin B(u^2_1,\pi/3)$, as otherwise
	we would be done by choosing either $u^1_1$ and $u^2_1$ or $u^3_1$ and $u^2_1$ (cf. Figure \ref{fig:angleregions2}).
	Finally, let us also suppose w.l.o.g.~that $(u^1_1)_2\geq (u^3_1)_2$.

  We now make a crucial observation:
	There must exist a point from $U_3$ in the arc of $\mathbb S^1$ containing $e_2$ and
	determined by $u^3_1$ and $-u^3_1$ (cf. green arc in Figure \ref{fig:angleregions3})
  \emph{different} from $u^3_1$. If this were not the case,
	then $0\notin\conv(U_3)$, a contradiction (again, cf.~Proposition \ref{prop:OOC}(4)). Assuming that this second point in this arc is $u^3_2$,
	we furthermore observe that it must necessarily belong to the arc of $\mathbb S^1$ determined by
	$-u^3_1$ and the left-most vertex of $B(u^2_1,\pi/3)$. Since $u^3_1$ lies in $\mathbb S^1$ between
	$e_1$ and the right-most vertex of $B(u^2_1,\pi/3)$, we see that
	the angle between $u^3_2$ (as well as $u^1_1$) and the left-most vertex of $B(u^2_1,\pi/3)$ is at most
	$\pi/3$. Thus we also have that $\arccos((u^3_2)^Tu^1_1)\leq\pi/3$, i.e., $(u^3_2)^Tu^1_1\geq 1/2$. Moreover,
	since $(u^3_2)_2\geq (-u^3_1)_2$ and $(u^1_1)_2\geq (u^3_1)_2$, we conclude that $(u^1_1+u^3_2)^Te_2=(u^1_1)_2+(u^3_2)_2\geq 0$, concluding the assertion
	for the choice $u^1_1$ and $u^3_2$. This concludes the proof of the claim, possibly changing naming between $U_2$ and $U_3$.
  \end{proof}
	
	For the equality case, we should have equality in all the inequalities above, in particular in (\ref{eq:ineq_case}), assuming w.l.o.g. that $l_1=l_2=l_3=1$. This means that
	$c=0$, $(u^1_1)^Tu^2_1=1/2$ (i.e.~$\arccos((u^1_1)^Tu^2_1)=\pi/3$), and $(u^1_1+u^2_1)^Tu^3_1=0$. If $u^1_1, u^2_1, u^3_1$ are located within
	$\mathbb S^1$ in clockwise order, then $\arccos((u^2_1)^Tu^3_1)=\pi/3=\arccos((u^1_1)^Tu^2_1)$, i.e., $u^1_1, u^2_1, u^3_1$ are vertices
	of a regular hexagon. Furthermore, the arc determined by $\pm u^3_1$ and containing $u^1_1$, $u^2_1$
	must contain a second point from $U_3$, say $u^3_2$. At the same time this point cannot be in the relative interior
	of this arc (as otherwise we would have that $(u^1_1+u^2_1)^Tu^3_2>0$, a contradiction with (\ref{eq:ineq_case})), hence showing $u^3_2=-u^3_1$.
	Again, we must have points from $U_1$ (say $u^1_2$) and from $U_2$ (say $u^2_2$)
	inside the other arc determined by $\pm u^3_1$.
	Again, for the same reason, $U_i\cap\relint(B(\pm u^3_1,\pi/3))=\emptyset$, $i=1,2$.
	Assuming that $u^3_1=e^1$, Proposition \ref{prop:OOC}(4) implies that $U_i\cap\{x\in\s^1:x_2\leq 0\}\neq\emptyset$, $i=1,2$.
	Then $U_i\cap\{x\in\s^1:x_2\leq-1/2\}\neq\emptyset$, $i=1,2$, namely, $u^i_2\in B(-e_2,\pi/6)$, $i=1,2$.

	Then we have that $(u^1_2)^Tu^2_2\geq 1/2$, with equality (as it actually happens) if and only if $u^1_2=(-1/2,\sqrt{3}/2)=-u^1_1$
	and $u^2_2=(-1/2,-\sqrt{3}/2)=-u^2_1$.
	Moreover, no other points on $U_1$, $U_2$, $U_3$ are allowed so that some of those conditions are violated.
	Therefore, $U_i=\{\pm u^i\}$, $i \in [3]$, where $\arccos((u^i)^Tu^j)=1/2$, for all $1\leq i<j\leq 3$,
	hence concluding the equality case and the theorem (cf. Figure \ref{fig:angleregions4} for a visualization).
\end{proof}

\begin{proof}[Proof of Corollary \ref{cor:radii_2_3sets}]
Since $\cir(K_i)=1$ with $K_i\subset\B^2_2$, let us denote by $u^j_i\in K_j\cap\mathbb S^1$,
for $i,j\in[3]$, the points given in Proposition \ref{prop:OOC}(3).
Denoting by $S_j:=\conv(\{u^j_i:i\in[3]\})\subset K_j$, $j\in[3]$,
due to the convexity of $K_j$ we have that
\[
\cir(K_1+K_2+K_3)\geq\cir(S_1+S_2+S_3).
\]
In order to prove the inequality, we will show that for every $c\in\R^2$,
then $S_1+S_2+S_3\nsubseteq c+\rho\B_2$, for any $\rho<2$. Applying Theorem \ref{thm:norm_2_3sets} to the sets
of vectors $U_j:=\{u^j_i:i\in[3]\}$, $j\in[3]$, we then deduce the existence of
indices $l_1,l_2,l_3\in[3]$, such that $\|u^1_{l_1}+u^2_{l_2}+u^3_{l_3}-c\|_2\geq 2$, hence implying the result.

For the equality case, we obtain in particular that if $K_1$, $K_2$, $K_3$ attain equality above, then the set of vectors
$U_1$, $U_2$, $U_3$ attain equality in Theorem \ref{thm:norm_2_3sets} too, hence obtaining the result.
\end{proof}

	\section{Generalized Minkowski spaces}
	
	\begin{lemma}\label{lem:circSumC}
		Let $K_i,C\in\K^n$, $i\in[j]$, $j\in[n]$. Then
		\[
		\max_{i=1,\dots,j}\cir(K_i,C)\leq\cir(K_1+\cdots+K_j,C).
		\]
		The inequality is sharp.
	\end{lemma}
	
	\begin{proof}
		Since $K_i\subset K_1+\cdots+K_j$ for every $i\in[j]$, we have
		$\cir(K_i,C)\leq\cir(K_1+\cdots+K_j,C)$ and thus
		\[
		\max_{i=1,\dots,j}\cir(K_i,C)\leq\cir(K_1+\cdots+K_j,C).
		\]
	\end{proof}


	\begin{proof}[Proof of Theorem \ref{thm:CircAdditionC}]
		Since $\cir(K_i,C)\leq\max_{i=1,\dots,j}\cir(K_i,C)$ for all $i \in [j]$ by Lemma \ref{lem:circSumC} we have
		\begin{equation} \label{eq:expanded_equality}
					\sum_{i=1}^j\cir(K_i,C)\leq j\,\max_{i=1,\dots,j}\cir(K_i,C)\leq j\,\cir(K_1+\cdots+K_j,C).
		\end{equation}

        We now show the equality case and begin with the \emph{only if} part. Assuming equality in \cref{eq:equality} we also obtain equality in \cref{eq:expanded_equality}.
        From this we conclude that $\cir(K_1,C)=\cdots=\cir(K_j,C)$.
        After suitable translations of $K_i$ and a rescalation of $C$,
        we can assume that $K_i \subset C$ and $\cir(K_i,C)=1$ for every $i\in[j]$. Hence together with the equality in \cref{eq:expanded_equality} we also obtain 
        $\cir(K_1+\cdots+K_j,C)=1$. Moreover, let $z_0\in\R^n$ be such that $K_1+\cdots+K_j\subset z_0+C$.

        Since $K_i \subset C$ and $\cir(K_i,C)=1$, 
        Proposition \ref{prop:OOC} implies
        the existence of points $p_i^l\in K_i\cap\bd(C)$ and vectors $a^l_i\in N(C,p^l_i)$,
        for $i\in[j]$, $l\in[k_i]$, $2\leq k_i\leq n+1$, such that
        \[
        0\in\conv(\{a^l_1:l\in[k_1]\})\cap\cdots\cap\conv(\{a^l_j:l\in[k_j]\}).
        \]
        Let us denote by $H_i:=\lin(\{a^l_i:l\in[k_i]\})$, for every $i\in[j]$.
        We will show that
        \[
        \left( K_1+\cdots+\widehat{K_i}+\cdots+K_j \right) \quad\bot\quad H_i,
        \]
        for every $i\in[j]$. For the sake of contradiction,
        let us suppose that
        \[
        \aff(K_1+\cdots+\widehat{K_{i}}+\cdots+K_j)\not\subset_t H_{i}^\bot,
        \]
        for some $i \in [j]$, i.e., there exist $x,y\in K_1+\cdots+\widehat{K_{i}}+\cdots+K_j$ such that
        the line $[x,y]\not\subset_t H_{i}^\bot$.


        In order to conclude the proof, since $[x,y]+K_{i}\subset K_1+\cdots+K_j\subset z_0+C$,
        we just need to show that $[x,y]+K_{i}\not\subset_tC$ holds. Indeed, assuming the contrary,
        we would have that
        \[
        \begin{split}
        [x,y]+K_{i} & \subset z+C \\
        & \subset z+H^\leq_{a^1_{i},(p^1_i)^Ta^1_{i}}\cap\cdots\cap H^\leq_{a^{k_{i}}_i,(p^{k_{i}}_i)^Ta^{k_{i}}_{i}},
        \end{split}
        \]
        for some $z\in H_i^\bot$. Hence
        \[
        (x+p^l_i)^Ta^l_i\leq (p^l_i)^Ta^l_i\quad\text{and}\quad(y+p^l_i)^Ta^l_i\leq (p^l_i)^Ta^l_i
        \]
        for every $l\in[k_i]$. This means that $x^Ta^l_i,y^Ta^l_i\leq 0$ for $l\in[k_i]$.
        Using the fact that $0\in\conv(\{a^l_i:l\in[k_i]\})$, we have
        $0=\sum_{l=1}^{k_i}\lambda_l a^l_i$ for suitable $\lambda_l>0$.
        Thus $0=\sum_{l=1}^{k_i}\lambda_l x^Ta^l_i$, which necessarily implies that
        $x^T a^l_i=0$ for all $l\in[k_i]$, and analogously $y^T a^l_i=0$ for all $l\in[k_i]$.
        Hence $x,y\in H_i^\bot$, a contradiction.


        Therefore $[x,y]+K_i\not\subset_tC$,
        implying the desired contradiction. Hence
        \[
        \aff(K_1+\cdots+\widehat{K_i}+\cdots+K_j)\subset_t H_i^\bot
        \]
        for every $i\in[j]$.
        Denoting by
        \[
        C_i:=H^\leq_{a^1_i,(p^1_i)^Ta^1_i}\cap\cdots\cap H^\leq_{a^{k_i}_i,(p^{k_i}_i)^Ta^{k_i}_i}
        \]
        which is a polytopal cylinder for every $i=1,\dots,j$, we have
        $\cir(K_i,C_i)=1$,
        and since $K_1+\cdots+K_i\subset z_0+C$ and $C\subset C_i$ for every $i=1,\dots,j$,
        then
        \[
        K_1+\cdots+K_j\subset z_0+C\subset z_0+C_1\cap\dots\cap C_j=(z_0+C_1)\cap\dots\cap(z_0+C_j),
        \]
        as desired.			

		For the \emph{if} part we obtain for all $i \in [j]$,
		\begin{align*}
			1 & \ge R(K_{1}+\dots+K_{n},\lambda C) \\
			& \ge R(K_{i},\lambda C) \\
			& \ge R(K_{i},C_{1} \cap \cdots \cap C_{n}) \\
			& \ge R(K_{i},C_{i}) = 1
		\end{align*}
		resulting in equality in the whole chain and in particular
		\[
		R(K_{1}+\dots+K_{n},\lambda C) =  R(K_{i},\lambda C)
		\]
		for all $i \in [j]$ which concludes the desired equality.
		
		\end{proof}
	
\section{Open questions}
	
    Denoting the \emph{$l_p$-unit ball} by $\B^n_p:=\{x\in\R^n:\|x\|_p^p=|x_1|^p+\cdots+|x_n|^p\leq 1\}$, for $p\geq 1$,
    then Theorem \ref{thrm:radii_2} implies that
    \[
    \cir(K_1+\cdots+K_j,\B^n_2)\geq\|(\cir(K_1,\B^n_2),\dots,\cir(K_j,\B^n_2))^T\|_2
    \]
    whereas Lemma \ref{lem:circSumC} implies that
    \[
    \cir(K_1+\cdots+K_j,\B^n_\infty)\geq\|(\cir(K_1,\B^n_\infty),\dots,\cir(K_j,\B^n_\infty))^T\|_\infty,
    \]
    and both inequalities are best possible. It is then quite tempting to conjecture the result below.
    \begin{conjecture}
    Let $K_i\in\K^n$, $i\in[j]$, and $p\in[2,\infty]$. Then
    \[
    \cir(K_1+\cdots+K_j,\B^n_p)\geq\left\|(\cir(K_1,\B^n_p),\dots,\cir(K_j,\B^n_p))^T\right\|_p.
    \]
    Moreover, the inequality is best possible, strengthening the general estimate given by Lemma \ref{lem:circSumC}.
    For instance, if $K_i:=[-e_i,e_i]$, $i\in[n]$, then
    \[
    \cir([-1,1]^n,\B^n_p)=n^{\frac{1}{p}}=\|(1,\dots,1)^T\|_p=\|(\cir(K_1,\B^n_p),\dots,\cir(K_n,\B^n_p))^T\|_p.
    \]
    \end{conjecture}
	
    It does not seem easy to guess the optimal bounds in extending Theorem \ref{thm:norm_2_3sets} to arbitrary $n\in\N$ and
    $i\geq n+1$. Nevertheless, we were told by Alexandr Polyanskii \cite{Po} (see also \cite{AN}) that the following conjecture for $i=n+1$ seems reasonable.
    \begin{conjecture}\label{conj:radiin+1}
    Let $U_i:=\{u_1^i,\dots,u_{k_i}^i\}\subset \S^{n-1}$, $i\in[n+1]$,
     $\lambda^i_1,\dots,\lambda^i_{k_i}>0$, $2\leq k_i\leq n+1$, and $c\in\R^n$, be such that
		\[
		0=\sum_{l=1}^{k_i}\lambda^i_lu^i_l,\quad\text{for every }i\in[n+1].
		\]
		Then
		\[
		\max_{l_1,\dots,l_{n+1}}\left\|u^1_{l_1}+\cdots+u^{n+1}_{l_{n+1}}-c\right\|_2\geq \sqrt{n+2}.
		\]
		Moreover, equality holds if and only if $c=0$ and, $U_i=\{\pm u^i\}$ $i=1,\dots,n+1$, 
and if $n$ is even then $u^1,\dots,u^{n+1}\in\S^{n-1}$ are the vertices of an n-dimensional regular simplex,
and if $n$ is odd, then $u^1,\dots,u^{2k+1}$ are the vertices of an $2k$-dimensional regular simplex, and
$u^{2k+2},\dots,u^{n+1}$ are some orthonormal set such that $u^i\bot u^j$ for every $1\leq i\leq 2k+1<j\leq n+1$.
    \end{conjecture}

    Let us observe that in the optimal case above, one should recognize that if $n$ is even then
    \[
    \left\|u^1+\cdots+u^{\frac{n}{2}+1}-u^{\frac{n}{2}+2}-\cdots-u^{n+1}\right\|_2=\sqrt{n+2}
    \]
    whereas, if $n$ is odd, then
    \[
    \left\|u^1+\cdots+u^{k+1}-u^{k+2}-\cdots-u^{2k+1}\pm u^{2k+2}\pm\cdots\pm u^{n+1}\right\|_2=\sqrt{n+2}.
    \]

    Moreover, Conjecture \ref{conj:radiin+1} might imply that for any $K_i\in\K^n$, $i \in [n+1]$, with
    $\cir(K_i)=1$ for every $i\in[n+1]$, we have
    \[
    \cir\left(\sum_{i=1}^{n+1}K_i\right)\geq \sqrt{n+2}.
    \]
    which again would be best possible.

	\emph{Acknowledgement} This paper started during the development of the lecture \emph{Selected topics
		on Convex and Discrete Geometry}, Wintersemester 2016/17, at the Zentrum Mathematik of the Technische Universit\"at M\"unchen (cf.~\cite{GM}).
The third author also thanks Technische Universit\"at M\"unchen and the University Centre of Defence of San Javier, where
part of this work has been done. Finally, we would like to thank the anonymous referee for helping us in improving several aspects of this paper.

\end{document}